\documentclass[11pt]{amsart}
\usepackage{mathrsfs}
\usepackage{amssymb}

\usepackage{titletoc}
\usepackage{amscd}
\usepackage{amsmath, amssymb}
\usepackage{amsfonts}
\usepackage[colorlinks,linkcolor=blue,citecolor=blue, pdfstartview=FitH]{hyperref}

\numberwithin{equation}{section}

\theoremstyle{plain}
\newtheorem{thm}{Theorem}[section]
\newtheorem{theorem}[thm]{Theorem}
\newtheorem{lemma}[thm]{Lemma}
\newtheorem{corollary}[thm]{Corollary}
\newtheorem{proposition}[thm]{Proposition}

\theoremstyle{definition}

\newtheorem{remark}[thm]{Remark}

\newtheorem{definition}[thm]{Definition}

\newtheorem{example}[thm]{Example}

\newtheorem{defn-thm}[thm]{Definition-Theorem}

\newcommand{\sO}{{\mathcal O}}

\newcommand{\bp}{\bar{\partial}}
\newcommand{\Om}{\Omega}

\newcommand{\ts}{\otimes}

\newcommand{\btheorem}{\begin{theorem}}
\newcommand{\etheorem}{\end{theorem}}
\newcommand{\bproposition}{\begin{proposition}}
\newcommand{\eproposition}{\end{proposition}}
\newcommand{\bdefinition}{\begin{definition}}
\newcommand{\edefinition}{\end{definition}}
\newcommand{\bcorollary}{\begin{corollary}}
\newcommand{\ecorollary}{\end{corollary}}
\newcommand{\bproof}{\begin{proof}}
\newcommand{\eproof}{\end{proof}}
\newcommand{\bremark}{\begin{remark}}
\newcommand{\eremark}{\end{remark}}
\newcommand{\eexample}{\end{example}}
\newcommand{\bexample}{\begin{example}}

\newcommand{\elemma}{\end{lemma}}
\newcommand{\blemma}{\begin{lemma}}

\newcommand{\p}{\partial}

\renewcommand{\bar}{\overline}

\renewcommand{\phi}{\varphi}

\newcommand{\ee}{\end{eqnarray*}}
\newcommand{\be}{\begin{eqnarray*}}

\newcommand{\beq}{\begin{equation}}
\newcommand{\eeq}{\end{equation}}

\newcommand{\bd}{\begin{enumerate}}
\newcommand{\ed}{\end{enumerate}}

\def\mb{\mathbb}
\def\mc{\mathcal}

\usepackage{fancyhdr}
\begin{document}

\title{Logarithmic vanishing theorems for effective $q$-ample divisors}
\makeatletter
\let\uppercasenonmath\@gobble
\let\MakeUppercase\relax
\let\scshape\relax
\makeatother
\author{Kefeng Liu}
\address{Kefeng Liu, Department of Mathematics,
Capital Normal University, Beijing, 100048, China}
\address{Department of Mathematics, University of California at Los Angeles, California 90095}

\email{liu@math.ucla.edu}
\author{Xueyuan Wan}
\address{Xueyuan Wan, Chern Institute of Mathematics \& LPMC, Nankai University, 300071}
\email{xywan@mail.nankai.edu.cn}

\author{Xiaokui Yang}
\address{Xiaokui Yang, Morningside Center of Mathematics, Academy of Mathematics and Systems Science\\ Chinese Academy of Sciences, Beijing, 100190, China}
\address{HCMS, CEMS, NCNIS, HLM, UCAS, Academy of Mathematics and Systems Science\\ Chinese Academy of Sciences, Beijing 100190, China}

 \email{xkyang@amss.ac.cn}
\begin{abstract}
Let $X$ be a compact K\"ahler manifold and $D$ be a simple normal
crossing divisor. If $D$ is the support of some effective $k$-ample
 divisor, we show
$$
    H^q(X,\Omega^p_X(\log D))=0,\quad \text{for}\quad p+q>n+k.$$ \end{abstract}

\maketitle

\section{Introduction}

 The classical Cartan--Serre--Grothendieck theorem says that  a line bundle $L$ over a compact complex manifold $X$ is
 \emph{ample}
  if and only if  for every coherent sheaf $\mathcal F$ on $X$, there exists a positive integer $m_0=m_0(X,\mathcal F,L)$ such that
\beq
 H^{q}(X,\mathcal F\otimes L^{\otimes m})=0, \quad \text{for}\quad q>0,\quad m\geq m_0. \label{asy}
\eeq On the other hand,  on compact complex manifolds, the ampleness
is also equivalent to the existence of a smooth metric with positive
curvature, thanks to the celebrated Kodaira embedding theorem.
Hence, the asymptotic vanishing theorem (\ref{asy}) can imply the
absolute Akizuki-Kodaira-Nakano vanishing theorem: \beq H^{q}(X,
\Om_X^p\ts L)=0,\quad  \text{for} \quad p+q>\dim X. \label{AKN} \eeq
It is well-known that  the implication of (\ref{AKN}) from
(\ref{asy}) is very deep and requires many fantastic  techniques in
complex geometry and algebraic geometry.\\

 The notion of ampleness  has a very natural generalization along the line of Cartan--Serre--Grothendieck criterion.

 \bdefinition A line bundle $L$ over a compact complex manifold $X$ is called \emph{$k$-ample},
 if for any coherent sheaf $\mathcal{F}$ on $X$ there exists a positive integer $m_0=m_0(X,L,\mathcal{F})>0$  such that
\beq H^q(X,\mathcal{F}\otimes L^m)=0,\quad \text{for}\quad q>k,\quad
m\geq m_0.\label{asy2}\eeq  \edefinition \noindent As analogous to
the ample line bundle case, one may wonder whether the asymptotic
vanishing theorem (\ref{asy2}) can imply certain
Akizuki-Kodaira-Nakano type vanishing theorems. With no doubt, it is
a very challenging problem, since in this general context, the
ambient manifold $X$ is not necessarily projective.\\

Actually, even if $X$ is projective,  the asymptotic vanishing
theorem (\ref{asy2}) can not imply the Kodaira type vanishing \beq
H^{q}(X, K_X\ts L)=0,\quad \text{for} \quad q>k. \label{AKN2} \eeq
Indeed, a counter-example is discovered by Ottem in
\cite[Section~9]{Ott}: there exist a projective threefold $X$ and a
$1$-ample line bundle $L$ such that $H^2(X,K_X\ts L)\neq 0$.\\

It is well-known that there is a differential geometric proof of the
vanishing theorem (\ref{AKN}) by using positive metrics on ample
line bundles. There is also a natural generalization for positive
line bundles. A line bundle $L$ over a compact complex manifold $X$
is called \emph{$k$-positive}, if there exists a smooth Hermitian
metric $h$ on $L$ such that the Chern curvature
$R^{(L,h)}=-\sqrt{-1}\p\bp\log h$ has at least $(n-k)$ positive
eigenvalues at any point on $X$. In $1962$, Andreotti and Grauert
proved in \cite[Th\'eor\`eme 14]{AG} that $k$-positive line bundles
are always $k$-ample (see also \cite[Proposition 2.1]{DPS}). In
\cite{DPS}, Demailly, Peternell and Schneider proposed a converse to
the Andreotti-Grauert theorem and asked whether  $k$-ample line
bundles are  $k$-positive. In dimension two, Demailly \cite{Dem}
proved an asymptotic version of a converse to the Andreotti-Grauert
theorem using tools related to asymptotic cohomology; subsequently,
Matsumura \cite{Mat} gave a positive answer to the question for
surfaces. Recently, the third author
 proved in \cite{Yang} that $(\dim X-1)$-ample line bundles are
$(\dim X-1)$-positive when $X$ is a projective manifold. However,
there exist higher dimensional counterexamples in the range
$\frac{\dim X}{2}-1<q<\dim X-2$, constructed by Ottem \cite{Ott}. In
summary, it should not be a reasonable approach to establishing
Akizuki-Kodaira-Nakano type vanishing theorems for $k$-ample line
bundles by using Hermitian
metrics.\\

\noindent The main result of our paper is

\begin{thm}\label{main}
    Let $X$ be a compact K\"ahler manifold with $\dim X=n$ and $D$ be a simple normal crossing divisor. If $D$ is the support of some effective $k$-ample divisor, then
\beq
    H^q(X,\Omega^p_X(\log D))=0,\quad \text{for}\quad p+q>n+k.
\eeq In particular, \beq\label{1.17}
    H^q(X,K_X\ts \sO_X(D))=0,\quad \text{for}\quad q>k.
\eeq
\end{thm}

\noindent Note that, when $X$ is projective and $D$ is an ample
divisor (i.e. $D$ is $0$-ample), Theorem \ref{main} is obtained by
Deligne in \cite{Del69}. For a general $k$-ample line bundle, the
special case (\ref{1.17}) is obtained by Greb and  K\"uronya in
\cite[Theorem~3.4]{Greb} when $X$ is projective. On a projective
toric variety $X$,  Broomhead,  Ottem and  Prendergast-Smith
 proved in \cite[Theorem 7.1]{Broo} that the $k$-ampleness of a line bundle $L$ can imply the Kodaira type vanishing theorem
$H^q(X,K_X\ts L)=0$ for  $q>k$.\\

%

\section{Logarithmic vanishing theorems}

Let $X$ be a compact complex manifold with $\dim_{\mb{C}}X=n$, and
$D=\sum_{i=1}^r D_i$
 be a simple normal crossing divisor, i.e. every irreducible component  $D_i$ is smooth and all intersections are
 transverse. That is, for every $p\in X$, we can
choose local coordinates $z_1, \cdots , z_n$ such that $D =
(\prod_{i=1}^k z_i  = 0)$ in a neighborhood of $p$.

 The sheaf of germs of differential
$p$-forms on $X$ with at most logarithmic poles along $D$, denoted
{$\Omega^p_X(\log D)$} (introduced by Deligne in \cite{Del69}) is
the sheaf whose sections on an open subset $V$ of $X$ are
\begin{align*}
\Gamma(V,\Omega_X^p(\log D)):=\{\alpha\in
\Gamma(V,\Omega_X^p\otimes\mc{O}_X(D))\ \text{and }d\alpha\in
\Gamma(V,\Omega_X^{p+1}\otimes\mc{O}_X(D)) \}.
\end{align*}

\noindent If $D'$ is an effective divisor with $\text{Supp}(D')=D$,
we denote by
\begin{align*}
    \Omega^p_X(*D')=\bigcup_{k\geq 0}\Omega^p_X(k D'),
\end{align*}
which is the sheaf of meromorphic $p$-forms that are holomorphic on
$X-D$ and have poles on of arbitrary (finite) order on $D$. Hence
\begin{align}\label{1.0}
\Omega^p_X(*D)=\Omega^p_X(*D').
\end{align}

\noindent A divisor $D'$ is called \emph{ $k$-ample} if
$\mc{O}_X(D')$ is a  $k$-ample line bundle. In the following result,
we see clearly how to use the $k$-ample condition in the proof of
Theorem \ref{main}.

\begin{lemma}\label{lemma}
Let $X$ be a compact complex manifold with $\dim_{\mb{C}}X=n$ and
$D'$ be an effective $k$-ample divisor. Then
\begin{align*}
H^q(X,\Omega^p_X(*D'))=0,\quad \text{for}\quad q>k.
\end{align*}
\end{lemma}
\begin{proof}
    We use the notations $[\bullet]$ and $[\bullet]_\ell$ to denote a class in $H^q(X,\Omega^p_X(*D'))$ and $H^q(X,\Omega^p_X(\ell D'))$ respectively.
     If $[\alpha]\in H^q(X,\Omega^p_X(*D'))$, from the definition of $*D'$, then there exists some $\ell_0>0$ such that
    \begin{align}\label{1.1}
    [\alpha]_{\ell}\in H^q(X,\Omega^p_X(\ell D')),\quad \text{for}\quad \ell\geq \ell_0.
    \end{align}
Since $\mc{O}_X(D')$ is $k$-ample, by definition, there exists some
$\ell_1>0$ such that
\begin{align}\label{1.2}
H^q(X,\Omega^p_X(\ell D'))=0,\quad \text{for}\quad q>k,\quad
\ell\geq \ell_1.
\end{align}

\noindent By (\ref{1.1}) and (\ref{1.2}), for $\ell\geq
\max\{\ell_0,\ell_1\}$ and $q>k$, we know
\begin{align*}
\alpha=\bp\beta
\end{align*}
for some $\beta\in A^{0,q}(X,\Omega^p_X(\ell D'))\subset
A^{0,q}(X,\Omega^p(*D'))$.  Therefore,
\begin{align*}
H^q(X,\Omega^p_X(*D'))=0
\end{align*} for $q>k$.
\end{proof}


 \noindent\emph{The
proof of Theorem \ref{main}.}  There is an  exact sequence:
\begin{align*}
H^q(X,\Omega^0_X(*D))\xrightarrow{d}H^q(X,\Omega^1_X(*D))\xrightarrow{d}\cdots\xrightarrow{d}
H^q(X,\Omega^n_X(*D)).
\end{align*} The associated  cohomology is denoted
by
 $$E_2^{p,q}=H^p_d(H^q(X,\Omega^*_X(*D))).$$
By Lemma \ref{lemma} and formula (\ref{1.0}), one has
\begin{align}\label{1.3}
    E_2^{p,q}=0,\quad \text{for}\quad q>k.
\end{align}

\noindent Let $\mb{H}^*(X, \Omega^*(*D))$ be the hypercohomology of
the complex. Hence, by (\ref{1.3}), one has \beq
    \mb{H}^s(X, \Omega^*(*D))=0,\quad\text{for}\quad s> n+k.\label{1.4}
\eeq

\noindent On the other hand, it is well-known (e.g. \cite[Page
453]{Griffith}) that  \beq \mb{H}^s(X, \Omega^*(*D))\cong
H^s(X-D,\mb{C}).\eeq  Hence, by (\ref{1.4}), one has
\begin{align}\label{1.41}
    H^s(X-D,\mb{C})=0,\quad\text{for}\quad s> n+k.
\end{align}
On the other hand, one has (e.g. \cite{Del72}) the following
identity when $X$ is K\"ahler:
    \begin{align}\label{1.5}
    \dim_{\mb{C}}H^s(X-D,\mb{C})= \sum_{p+q=s}\dim_{\mb{C}}H^q(X,\Omega^p_X(\log D)).
    \end{align}
Combining (\ref{1.41}) with (\ref{1.5}), we obtain
\begin{align*}
    H^q(X,\Omega^p_X(\log D))=0,
\end{align*} for $p+q> n+k$.
\qed

\end{document}